\documentclass{amsart}

\usepackage{amssymb}
\usepackage{amsmath}
\usepackage{graphicx}

\newcommand{\erase}[1]{}

\newtheorem{theorem}{Theorem}[section]
\newtheorem{lemma}[theorem]{Lemma}
\newtheorem{proposition}[theorem]{Proposition}
\newtheorem{corollary}[theorem]{Corollary}

\newtheorem{_definition}[theorem]{Definition}
\newenvironment{definition}{\begin{_definition}\rm}{\end{_definition}}

\newtheorem{_remark}[theorem]{\it Remark}
\newenvironment{remark}{\begin{_remark}\rm}{\end{_remark}}

\numberwithin{equation}{section}
\numberwithin{table}{section}
\numberwithin{figure}{section}

\newcommand{\F}{\mathord{\mathbb F}}
\renewcommand{\P}{\mathord{\mathbb  P}}

\newcommand{\maprightsp}[1]{\; \smash{\mathop{\; \longrightarrow \; }\limits\sp{#1}}\; }

\newcommand{\mapdown}{\phantom{\Big\downarrow}\hskip -8pt \downarrow}
\newcommand{\mapdownright}[1]{\mapdown\rlap{$\vcenter{\hbox{$\scriptstyle#1$}}$}}
\newcommand{\mapdownleft}[1]{\llap{$\vcenter{\hbox{$\scriptstyle#1$}}$}%
\mapdown}
\newcommand{\mapdownsurj}{
\hbox{$\bigm\downarrow$}
\llap{\hbox{\raise 2pt\hbox{$\bigm\downarrow$}}}%
\vstrechmapdown
}

\newcommand{\inj}{\hookrightarrow}

\newcommand{\set}[2]{\{\; {#1} \; \mid \; {#2} \;  \}}

\newcommand{\sprime}{\sp\prime}
\newcommand{\spar}[1]{\sp{(#1)}}

\newcommand{\dual}{\sp{\vee}}
\newcommand{\inv}{\sp{-1}}

\newcommand{\PGL}{\mathord{\mathrm {PGL}}}

\newcommand{\Aut}{\operatorname{\mathrm {Aut}}\nolimits}

\newcommand{\Sing}{\operatorname{\mathrm {Sing}}\nolimits}
\newcommand{\Pic}{\operatorname{\mathrm {Pic}}\nolimits}
\newcommand{\NS}{\operatorname{\mathrm {NS}}\nolimits}

\newcommand{\rmand}{\textrm{and}}
\newcommand{\quand}{\quad\rmand\quad}

\newcommand{\Fq}{\F_{q}}
\newcommand{\Fqq}{\F_{q^2}}
\newcommand{\tilS}{\tilde{S}}
\newcommand{\tilM}{\tilde{M}}
\newcommand{\tileta}{\tilde{\eta}}
\newcommand{\tilSigma}{\tilde{\Sigma}}
\newcommand{\tilB}{\tilde{B}}
\newcommand{\tilz}{\tilde{z}}

\newcommand{\clM}{\overline{M}}

\newcommand{\intnumb}[1]{\langle #1 \rangle}

\begin{document}

\title[Ballico-Hefez  curves]
{On Ballico-Hefez  curves and associated supersingular surfaces}

\author{Hoang Thanh Hoai}
\address{
Department of Mathematics, 
Graduate School of Science, 
Hiroshima University,
1-3-1 Kagamiyama, 
Higashi-Hiroshima, 
739-8526 JAPAN
}
\email{hoangthanh2127@yahoo.com}

\author{Ichiro Shimada}
\address{
Department of Mathematics, 
Graduate School of Science, 
Hiroshima University,
1-3-1 Kagamiyama, 
Higashi-Hiroshima, 
739-8526 JAPAN
}
\email{shimada@math.sci.hiroshima-u.ac.jp}

\begin{abstract}
Let $p$ be a prime integer, and $q$ a power of $p$.
The Ballico-Hefez  curve is a  non-reflexive nodal rational plane curve of degree $q+1$
in characteristic $p$.
We investigate its automorphism group and  defining equation.
We also prove that the surface obtained as 
the cyclic cover of the projective plane branched along the Ballico-Hefez  curve
is unirational, and hence is supersingular.
As an application,
we obtain  a new projective model of the supersingular $K3$ surface 
with Artin invariant $1$ in characteristic $3$ and $5$.
\end{abstract}

\thanks{Partially supported by
JSPS Grant-in-Aid for Challenging Exploratory Research No.23654012
and 
JSPS Grants-in-Aid for Scientific Research (C) No.25400042 
}



\maketitle

\section{Introduction}\label{sec:Intro}
We work over an algebraically closed field $k$  of positive characteristic $p>0$.
Let $q=p^{\nu}$ be a power of $p$.
\par
\medskip
In positive characteristics,
algebraic varieties often possess interesting properties
that are not observed in characteristic zero.
%
%
One of those properties is the failure of reflexivity.
In~\cite{MR1092144},
Ballico and Hefez classified irreducible plane curves $X$ of degree $q+1$ 
such that the natural morphism from the  conormal variety $C(X)$ of $X$
to the dual curve $X\dual$ has inseparable degree $q$.
The Ballico-Hefez  curve in the title of this note is one of
the curves that appear in their classification.
It is  defined in Fukasawa, Homma and Kim~\cite{MR2961398} as follows.
\begin{definition}
The \emph{Ballico-Hefez  curve}
is   the image 
of the morphism
$\phi: \P^1\to \P^2$ defined  by
$$
[s:t]\mapsto [s^{q+1}: t^{q+1}: s t^q + s^q t].
$$
\end{definition}
\begin{theorem}[Ballico and Hefez~\cite{MR1092144}, Fukasawa, Homma and Kim~\cite{MR2961398}]
{\rm (1)} Let $B$ be the Ballico-Hefez  curve.
Then $B$ is a  curve of degree $q+1$ with $(q^2-q)/2$ ordinary nodes,  the dual curve $B\dual$ is of degree $2$, and
the natural morphism $C(B)\to B\dual$ has inseparable degree $q$.
\par
{\rm (2)} 
Let $X\subset \P^2$ be an irreducible  singular curve of degree $q+1$
such that the dual curve $X\dual$ is of degree $>1$ and 
the natural morphism $C(X)\to X\dual$ has inseparable degree $q$.
Then $X$ is projectively isomorphic to 
the Ballico-Hefez curve.
\end{theorem}
Recently,
 geometry and arithmetic of the Ballico-Hefez  curve 
have been investigated by
Fukasawa, Homma and Kim~\cite{MR2961398} and Fukasawa~\cite{Fukasawa2}
from various points of view,
including coding theory and Galois points.
As is pointed out in~\cite{MR2961398},
the Ballico-Hefez curve has many properties in common with the Hermitian curve;
that is, the Fermat curve 
of degree $q+1$,
which also appears in the classification of Ballico and Hefez~\cite{MR1092144}.
In fact, we can easily see that the image of the line 
$$
x_0+x_1+x_2=0
$$
in $\P^2$ by the morphism $\P^2\to \P^2$ given by
$$
[x_0: x_1: x_2]\mapsto [x_0^{q+1}: x_1^{q+1}: x_2^{q+1}]
$$
is projectively isomorphic to the Ballico-Hefez  curve.
Hence,
up to  linear transformation of coordinates,  
the Ballico-Hefez  curve is  defined  by an equation
$$
x_0^\frac{1}{q+1}+ x_1^\frac{1}{q+1}+x_2^\frac{1}{q+1}=0
$$
in the style of ``Coxeter curves"~(see Griffith~\cite{MR717305}). 
%
\par
\medskip
In this note, we prove the the following:
%
\begin{proposition}\label{prop:aut}
Let $B$ be the Ballico-Hefez  curve.
Then  the group 
$$
\Aut(B):=\set{g\in \PGL_3(k)}{g(B)=B}
$$
 of  
projective automorphisms  of $B\subset \P^2$
is isomorphic to 
$\PGL_2(\Fq)$.
\end{proposition}
\begin{proposition}\label{prop:defeq}
The Ballico-Hefez  curve  is defined by the following equations:
\begin{itemize}
\item When $p=2$, 
$$
x_0^q x_1+x_0 x_1^q +x_2^{q+1}+\sum_{i=0}^{\nu-1} x_0^{2^{i}} x_1^{2^i} x_2^{q+1-2^{i+1}} =0,
\quad\textrm{where $q=2^\nu$}.
$$
\item When $p$ is odd,
$$
2(x_0^q x_1+x_0 x_1^q)-x_2^{q+1}-(x_2^2-4 x_1 x_0)^{\frac{q+1}{2}}=0. 
$$
\end{itemize}
\end{proposition}
\begin{remark}
In fact, the defining equation for $p=2$  has been obtained by Fukasawa
in an apparently different form
(see Remark 3 of~\cite{Fukasawa1}).
\end{remark}
Another property of algebraic varieties peculiar to positive characteristics 
is the failure of L\"uroth's theorem for surfaces;
a non-rational surface can be unirational in positive characteristics.
A famous example of this phenomenon is the Fermat surface  of degree $q+1$.
Shioda~\cite{MR0374149} and Shioda-Katsura~\cite{MR526513} showed that
the Fermat surface $F$ of degree $q+1$ is unirational (see also~\cite{MR1176080} for another proof).
This  surface $F$ is 
obtained as the cyclic cover of $\P^2$ with degree $q+1$ branched along
the Fermat curve  of degree $q+1$,
and hence, for any divisor $d$ of $q+1$,
the cyclic cover of $\P^2$ with degree $d$ branched along 
the Fermat curve  of degree $q+1$ is also unirational.
\par
\medskip
We  prove an analogue of this result for the Ballico-Hefez curve.
Let $d$ be a divisor of $q+1$ larger than $1$.
Note that $d$ is prime to $p$.
\begin{proposition}\label{prop:unirational}
Let $\gamma: S_d\to \P^2$ be 
the cyclic covering of $\P^2$ with degree $d$ branched along 
the Ballico-Hefez curve.
Then there exists a dominant rational map
$\P^2\cdots \to S_d$ of degree $2q$ with inseparable degree $q$.
\end{proposition}
Note that  $S_d$ is not rational
except for the case $(d, q+1)=(3,3)$ or $ (2, 4)$.
\par
\medskip
A smooth surface $X$ is said to be \emph{supersingular} (in the sense of Shioda)
if the second $l$-adic cohomology group $H^2(X)$ of $X$ is generated by the classes of curves.
Shioda~\cite{MR0374149} proved that every smooth unirational surface is supersingular.
Hence we obtain the following:
\begin{corollary}\label{cor:supersingular}
Let $\rho : \tilS_d \to S_d$ be the minimal resolution of $S_d$.
Then the  surface $\tilS_d$  is supersingular.
\end{corollary}
We present a finite set of curves on $\tilS_d$ whose classes span  $H^2(\tilS_d)$.
For a point $P$ of $\P^1$,
let $l_P\subset \P^2$ denote the line tangent at $\phi(P)\in B$
to the branch of $B$ corresponding to $P$.
It was shown in~\cite{MR2961398} that,
if $P$ is an $\Fqq$-rational point of $\P^1$,
then $l_P$ and $B$ intersect only at $\phi(P)$,
and hence the strict transform  of $l_P$ by the 
composite $\tilS_d \to S_d\to \P^2$ is a union of $d$ rational curves $l_P\spar{0}, \dots, l_P\spar{d-1}$.
\begin{proposition}\label{prop:supersingular}
The cohomology group 
 $H^2(\tilS_d)$ is generated by the classes of the following rational curves on $\tilS_d$;
the irreducible components 
of the exceptional divisor of the resolution $\rho: \tilS_d\to S_d$ and
the rational curves $l_P\spar{i}$, where $P$ runs through the set $\P^1(\Fqq)$
of $\Fqq$-rational points of $\P^1$
and $i=0, \dots, d-1$.
\end{proposition}
Note that, when $(d, q+1)=(4,4)$ and $(2, 6)$,
the surface $\tilS_d$ is a $K3$ surface.
In these cases,
we can prove that the classes of rational curves given  in Proposition~\ref{prop:supersingular}
generate the N\'eron-Severi lattice $\NS(\tilS_d)$ of $\tilS_d$,
and that the discriminant of $\NS(\tilS_d)$ is $-p^{2}$. 
Using this fact and 
the result of Ogus~\cite{MR563467, MR717616} and Rudakov-Shafarevich~\cite{MR633161}
on the uniqueness of a supersingular $K3$ surface with Artin invariant $1$,
we prove the following:
%
\begin{proposition}\label{prop:ssK3}
{\rm (1)} If $p=q=3$, then $\tilS_4$ is isomorphic to the Fermat quartic surface
$$
 w^4+x^4+y^4+z^4=0.
 $$
{\rm (2)} If $p=q=5$, then $\tilS_2$ is isomorphic to the Fermat sextic double plane
$$
w^2=x^6+y^6+z^6.
 $$
\end{proposition}
Recently, many studies on these supersingular $K3$ surfaces with Artin invariant $1$
in characteristics $3$ and $5$ have been carried out. 
See~\cite{MR2862188, KondoShimada} for characteristic $3$ case,
and~\cite{KKSchar5, shimadapreprintchar5} for characteristic $5$ case.
\par
\medskip
Thanks are due to Masaaki Homma and Satoru Fukasawa for their comments.
We also thank the referee for  his/her suggestion on the first version of this paper.
\section{Basic properties of the Ballico-Hefez curve}
We recall some properties of the Ballico-Hefez curve $B$.
See Fukasawa, Homma and Kim~\cite{MR2961398} for the proofs.
\par
\medskip
It is easy to see that the morphism $\phi:\P^1\to \P^2$ is birational 
onto its image $B$, and that
the degree of the plane curve $B$ is $q+1$.
The singular locus $\Sing (B)$ of $B$ consists of $(q^2-q)/2$ ordinary nodes,
and we have 
$$
\phi\inv(\Sing(B))=\P^1(\Fqq)\setminus \P^1(\Fq).
$$
In particular,  the singular locus $\Sing(S_d)$ of $S_d$ consists of $(q^2-q)/2$ 
ordinary rational double points of type $A_{d-1}$.
Therefore,  
by Artin~\cite{MR0146182, MR0199191},
the surface  $S_d$ is not rational
if  $(d, q+1)\ne(3,3), (2, 4)$.
\par
\medskip
Let $t$ be the affine coordinate of $\P^1$ obtained from $[s:t]$ by putting  $s=1$,
and let $(x, y)$ be the affine coordinates of $\P^2$ such that $[x_0: x_1: x_2]=[1:x:y]$.
Then the morphism $\phi:\P^1\to\P^2$ is given by 
$$
t\mapsto (t^{q+1}, t^q+t).
$$
For a point $P=[1:t]$ of $\P^1$,
the line $l_P$ is defined by
$$
x-t^q y +t^{2q}=0.
$$

Suppose that $P\notin \P^1(\Fqq)$.
Then $l_P$ intersects $B$ at $\phi(P)=(t^{q+1}, t^q+t)$
with multiplicity $q$ and at the point $(t^{q^2+q}, t^{q^2}+t^q)\ne \phi(P)$
with multiplicity $1$.
In particular, we have $l_P\cap \Sing(B)=\emptyset$.
 
Suppose that $P\in \P^1(\Fqq)\setminus\P^1(\Fq)$.
Then $l_P$ intersects $B$ at the node $\phi(P)$ of $B$
with multiplicity $q+1$.
More precisely,
$l_P$ intersects the branch of $B$  corresponding to $P$
with multiplicity $q$,
and the other branch transversely.

Suppose that $P\in \P^1(\Fq)$.
Then $\phi(P)$ is a smooth point of $B$,
and $l_P$ intersects $B$ at $\phi(P)$
with multiplicity $q+1$.
In particular, we have $l_P\cap \Sing(B)=\emptyset$.

Combining these facts, we see that 
$\phi(\P^1(\Fq))$ coincides with the set of 
smooth inflection points of $B$.
(See~\cite{MR2961398} for the definition of inflection points.)

\section{Proof of Proposition~\ref{prop:aut}}

We denote by $\phi_B: \P^1\to B$ the birational morphism  $t\mapsto (t^{q+1}, t^q+t)$
from $\P^1$ to $B$.
We identify $\Aut(\P^1)$ with   $\PGL_2(k)$
by letting  $\PGL_2(k)$ act on $\P^1$ by
$$
[s:t]\mapsto [as +bt: cs +dt]\quad\textrm{for}\quad \left[\begin{array}{cc} a & b\\ c& d \end{array}\right] \in \PGL_2(k).
$$
Then $\PGL_2(\Fq)$ is the subgroup of $\PGL_2(k)$ consisting of elements 
that leave the set $\P^1(\Fq)$  invariant.
Since $\phi_B$ is birational,
the projective automorphism group $\Aut(B)$ of $B$ acts on $\P^1$ via $\phi_B$.
The subset  $\phi_B(\P^1(\Fq))$ of $B$ is projectively characterized as
the set of smooth inflection points of $B$,
and we have $\P^1(\Fq)=\phi_B\inv(\phi_B(\P^1(\Fq)))$.
Hence 
$\Aut(B)$ is contained in  the subgroup $\PGL_2(\Fq)$ of $\PGL_2(k)$.
Thus, in order to prove Proposition~\ref{prop:aut},
it is enough to show that
every element
$$
 g:=\left[\begin{array}{cc} a & b\\ c& d \end{array}\right] 
 \quad\textrm{with}\quad a,b,c,d\in \Fq
$$
of $\PGL_2(\Fq)$ is coming from the action of an element of $\Aut(B)$.
We put
$$
\tilde{g}:=\left[\begin{array}{ccc} a^2 & b^2 & ab \\ c^2 & d^2 & cd \\ 2ac & 2bd & ad+bc \end{array}\right],
$$
and let the matrix $\tilde{g}$ act on $\P^2$ by the left multiplication 
on the column vector ${}^{t}[x_0: x_1:x_2]$.
Then we have
$$
\phi\circ g=\tilde{g}\circ \phi, 
$$ 
because we have $\lambda^q=\lambda$ for $\lambda=a,b,c,d\in \Fq$.
Therefore $g\mapsto \tilde{g} $ gives an isomorphism
from $\PGL_2(\Fq)$ to $\Aut(B)$.
\section{Proof of Proposition~\ref{prop:defeq}}
We put
$$
F(x, y):=\begin{cases}
x+x^{q}+y^{q+1}+\sum_{i=0}^{\nu-1}x^{2^i}y^{q+1-2^{i+1}} & \textrm{if $p=2$ and $q=2^\nu$,} \\
2x+2x^q-y^{q+1}-(y^2-4x)^{\frac{q+1}{2}} & \textrm{if $p$ is odd,} 
\end{cases}
$$
that is, $F$ is obtained from the homogeneous polynomial in Proposition~\ref{prop:defeq}
by putting $x_0=1, x_1=x, x_2=y$.
Since the polynomial $F$ is of degree $q+1$ and 
the plane curve $B$ is also of degree $q+1$,
it is enough to show that 
$F(t^{q+1}, t^q+t)=0$.
\par
\medskip
Suppose that $p=2$ and $q=2^\nu$.
We put 
 $$
 S(x,y):=\sum_{i=0}^{\nu-1}\left(\frac{x}{y^2}\right)^{2^i}.
 $$
Then  $S(x,y)$ is  a root of the Artin-Schreier equation 
$$ 
s^2+s=\left(\frac{x}{y^2}\right)^q+\frac{x}{y^2}.
$$
Hence $S_1:=S(t^{q+1},t^q+t)$
 is a root of the equation $s^2+s=b$,
where 
$$
b:=\left[\frac{t^{q+1}}{(t^q+t)^2}\right]^q+\frac{t^{q+1}}{(t^q+t)^2}=\frac{t^{2q^2+q+1}+t^{q^2+3q}+t^{q^2+q+2}+t^{3q+1}}{(t^q+t)^{2q+2}}.
$$
We put  
$$
S\sprime(x,y):=\frac{x+x^q+y^{q+1}}{y^{q+1}}. 
$$
We can verify that $S_2:=S\sprime(t^{q+1},t^q+t)$ is also
a root of the equation $s^2+s=b$. 
Hence we have either $S_1=S_2$ or $S_1=S_2+1$.
We can easily see that both of 
the rational functions $S_1$ and $S_2$ on $\P^1$ have zero at $t=\infty$.
Hence $S_1=S_2$ holds, from which we obtain
 $F(t^{q+1},t^q+t)=0$.
\par
\medskip
Suppose that $p$ is odd.
We put
\begin{eqnarray*}
&&S(x,y):=2x+2x^q-y^{q+1}, \quad S_1:=S(t^{q+1},t^q+t), \quand \\
&&S\sprime(x,y):=(y^2-4x)^{\frac{q+1}{2}},  \quad S_2:=S\sprime(t^{q+1},t^q+t).
\end{eqnarray*}
Then it is easy to verify that both of $S_1^2$ and $S_2^2$ are equal to 
$$
t^{2q^2+2q}-2t^{2q^2+q+1}+t^{2q^2+2}-2t^{q^2+3q}+4t^{q^2+2q+1}-2t^{q^2+q+2}+t^{4q}-2t^{3q+1}+t^{2q+2}.
$$
Therefore either $S_1=S_2$ or $S_1=-S_2$
holds.
Comparing the coefficients of the top-degree terms of 
the polynomials $S_1$ and $S_2$ of $t$,
we see that $S_1=S_2$,
whence  $F(t^{q+1},t^q+t)=0$ follows.
\section{Proof of Propositions~\ref{prop:unirational} and~\ref{prop:supersingular}}
We consider the universal family 
$$
 L :=\set{(P, Q)\in \P^1\times\P^2}{Q\in l_P}
$$
of the lines $l_P$, 
which is defined by 
$$
x-t^q y +t^{2q}=0
$$
 in  $\P^1\times\P^2$, and let
$$
\pi_1:  L \to\P^1, \quad
\pi_2:  L \to \P^2
$$
be the projections.
We see that  $\pi_1: L\to \P^1$ has two sections
\begin{eqnarray*}
\sigma_1 &:& t\mapsto (t, x, y)=(t, t^{q+1}, t^q+t),\\
\sigma_q &:& t\mapsto (t, x, y)=(t, t^{q^2+q}, t^{q^2}+t^q).
\end{eqnarray*}
For $P\in \P^1$, we have $\pi_2(\sigma_1(P))=\phi(P)$ and 
$l_P\cap B=\{\pi_2(\sigma_1(P)), \pi_2(\sigma_q(P))\}$.
Let $\Sigma_1\subset L$ and $\Sigma_q\subset L$ denote the images of $\sigma_1$
and $\sigma_q$, respectively.
Then $\Sigma_1$ and $\Sigma_q$ are smooth curves,
and they intersect transversely.
Moreover, their intersection points are contained in $\pi_1\inv (\P^1(\Fqq))$.
\par
\medskip
We denote by  $\clM$ the fiber product 
of $\gamma: S_d\to\P^2$ and $\pi_2:L\to \P^2$ over $\P^2$.
The pull-back   $\pi_2^* B$  of  
$B$  by $\pi_2$ is equal to the divisor $q\Sigma_1+\Sigma_q$.
Hence $\clM$ is defined by
\begin{equation}\label{eq:clM}
\begin{cases}
z^d=(y-t^q-t)^q (y-t^{q^2}-t^q), & \\
x-t^q y +t^{2q}=0.
\end{cases}
\end{equation}
We denote by  $M\to \clM$  the normalization, 
and  by
$$
\alpha : M \to L, \quad \eta: M\to S_d
$$
the natural projections.
Since $d$ is prime to $q$,
the cyclic covering   $\alpha: M\to  L $ of degree $d$
branches exactly  along 
the curve $\Sigma_1\cup \Sigma_q$.
Moreover, 
 the singular locus $\Sing(M)$ of $M$ is located over $\Sigma_1\cap \Sigma_q$,
 and hence is contained in
$\alpha\inv(\pi_1\inv (\P^1(\Fqq)))$.
\par
\medskip
Since $\eta$ is dominant and $\rho: \tilS_d\to S_d$ is birational,
 $\eta$ induces a rational map
$$
\eta\sprime: M\cdots\to \tilS_d.
$$
Let $A$ denote the affine open curve $\P^1\setminus \P^1(\Fqq)$.
We put
$$
L_A:=\pi_1\inv(A),
\quad
M_A:=\alpha\inv(L_A).
$$
Note that $M_A$ is smooth.
Let $\pi_{1, A}: L_A\to A$ and $\alpha_A: M_A\to L_A$ be the restrictions of $\pi_1$ and $\alpha$,
respectively.
If $P\in A$, then  $l_P$ is disjoint from $\Sing(B)$,
and hence $\eta(\alpha\inv(\pi_1\inv(P)))=\gamma\inv(l_P)$ is disjoint from $\Sing(S_d)$.
Therefore the restriction of $\eta\sprime$ to $M_A$ is a morphism.
It follows that we have a proper birational morphism
$$
\beta: \tilM\to M
$$
from a smooth surface $\tilM$ to $M$
such that $\beta$ induces an isomorphism from $\beta\inv(M_A)$ to $M_A$
and that the rational map $\eta\sprime$ extends to 
a morphism $\tileta: \tilM\to \tilS_d$.
Summing up,
we obtain the following commutative diagram:
\begin{equation}\label{eq:diagram}
\begin{array}{ccccc}
M_A & \inj  & \tilM &\maprightsp{\tileta} & \tilS_d \\
|| & \square &\mapdownright{\beta} & & \mapdownright{\rho} \\
M_A & \inj & M & \maprightsp{\eta} & S_d \\
\mapdownleft{\alpha_A} & \square & \mapdownright{\alpha} &  & \mapdownright{\gamma} \\
L_A & \inj & L & \maprightsp{\pi_2} & \P^2 \\
\mapdownleft{\pi_{1, A}} & \square & \mapdownright{\pi_1} & & \\
A & \inj & \P^1 &\hskip -20pt .&
\end{array}
\end{equation}
\par
\medskip
Since the defining equation 
$x-t^q y +t^{2q}=0$
 of $ L $ in $\P^1\times \P^2$
 is a polynomial in $k[x, y][t^q]$,
 and its discriminant as a quadratic equation of $t^q$ is $y^2-4x\ne 0$, 
the projection $\pi_2$ is a finite morphism  of degree $2q$ and its inseparable degree is $q$.
Hence  $\eta$ is also a finite morphism 
of degree $2q$ and its inseparable degree is $q$.
Therefore, in order to prove Proposition~\ref{prop:unirational},
 it is enough to show that $M$ is rational.
 We denote by $k(M)=k(\clM)$ the function field of $M$.
 Since $ x=t^qy-t^{2q}$ on $\clM$, the field $k(M)$ 
 is generated over $k$ by $y, z$ and $t$.
 Let $c$ denote the integer $(q+1)/d$, and put
 $$
 \tilz:= \frac{z}{(y-t^q-t)^{c}} \in k(M).
 $$
 Then, from the defining equation~\eqref{eq:clM} of $\clM$, we have
 $$
 \tilz^d=\frac{y-t^{q^2}-t^q}{y-t^q-t}.
 $$
 Therefore we have 
 $$
 y=\frac{ \tilz^d(t^q+t)-(t^{q^2}+t^q)}{ \tilz^d-1},
 $$
 and hence $k(M)$ is equal to the purely transcendental extension $k(\tilz, t)$ of $k$.
Thus Proposition~\ref{prop:unirational} is proved.
\par
\medskip
We put
$$
\Xi:=\tilM\setminus M_A=\beta\inv(\alpha\inv(\pi_1\inv(\P^1(\Fqq)))).
$$
Since the cyclic covering $\alpha: M\to L$ branches along the curve $\Sigma_1=\sigma_1(\P^1)$,
the section $\sigma_1: \P^1\to L$ of $\pi_1$
lifts to a section $\tilde\sigma_1: \P^1\to M$ of $\pi_1\circ \alpha$.
Let $\tilSigma_1$ denote the strict transform of 
the image of  $\tilde\sigma_1$ by $\beta:\tilM\to M$.
%
\begin{lemma}\label{lem:Pic}
The Picard group $\Pic (\tilM)$ of $\tilM$ is generated by 
the classes of $\tilSigma_1$ and the irreducible components of   $\Xi$.
\end{lemma}
\begin{proof}
Since    $\Sigma_1\cap \Sigma_q\cap L_A=\emptyset$, 
the morphism 
$$
\pi_{1, A}\circ \alpha_A : M_A \to A
$$
is a smooth $\P^1$-bundle.
Let $D$ be an irreducible curve on $\tilM$,
and let $e$ be the degree of
$$
\pi_1\circ\alpha\circ \beta|_D : D\to \P^1.
$$
Then the divisor $D-e\tilSigma_1$ on $\tilM$ is of degree $0$
on the general  fiber of the  smooth $\P^1$-bundle $\pi_{1, A}\circ \alpha_A$.
Therefore $(D-e\tilSigma_1)|_{M_A}$ is linearly equivalent in $M_A$ to a multiple of 
a fiber  of  $\pi_{1, A}\circ \alpha_A$.
Hence $D$ is linearly equivalent to a linear combination of $\tilSigma_1$
and irreducible curves in 
the boundary  $\Xi=\tilM\setminus M_A$.
\end{proof}
The rational curves on $\tilS_d$ listed in Proposition~\ref{prop:supersingular}
are exactly equal to  the irreducible components of
$$
\rho\inv (\gamma\inv (\bigcup_{P\in \P^1(\Fqq)} l_P)).
$$
Let $V\subset H^2(\tilS_d)$ denote the linear subspace spanned by 
the classes of these rational curves.
We will show that $V=H^2(\tilS_d)$.
\par
\medskip
Let $h\in H^2(\tilS_d)$
denote the class of the pull-back
of a  line of $\P^2$ by the morphism $\gamma\circ \rho: \tilS_d\to \P^2$.
Suppose that  $P\in \P^1(\Fq)$. Then $l_P$ is disjoint from $\Sing(B)$.
Therefore we have 
$$
h=[(\gamma\circ \rho)^*(l_P)]=[l_P\spar{0}]+\cdots +[l_P\spar{d-1}] \in V.
$$
Let $\tilB$  denote 
the   strict transform of $B$ by $\gamma\circ \rho$.
Then $\tilB$  is written as $d\cdot R$,
where $R$ is a reduced curve on $\tilS_d$ whose support is equal to $\tileta(\tilSigma_1)$.
On the other hand,
the class of the total transform $(\gamma\circ \rho)^*B$ of $B$ by $\gamma\circ \rho$
is equal to $(q+1)h$.
%
Since the difference of the divisors $d\cdot R$ and $(\gamma\circ \rho)^*B$
is a linear combination of exceptional curves of $\rho$, 
 we have 
\begin{equation}\label{eq:imSigma}
\tileta_* ([\tilSigma_1]) \in V.
\end{equation}
By the commutativity of the diagram~\eqref{eq:diagram},
we have
\begin{equation*}\label{eq:imXi}
\tileta (\Xi) \;\;\subset\;\; \rho\inv (\gamma\inv (\bigcup_{P\in \P^1(\Fqq)} l_P)).
\end{equation*}
Hence, for any irreducible component $\Gamma$ of $\Xi$,
we have 
\begin{equation}\label{eq:imGamma}
\tileta_* ([\Gamma]) \in V.
\end{equation}
Let $C$ be an arbitrary  irreducible curve on $\tilS_d$.
Then we have
$$
\tileta_*\tileta^* ([C])=2q [C].
$$
By Lemma~\ref{lem:Pic},
there exist integers $a$, $b_1, \dots, b_m$ and irreducible components
$\Gamma_1, \dots, \Gamma_m$ of $\Xi$
such that 
the divisor $\eta^*C$ of $\tilM$ is linearly equivalent to
$$
a \tilSigma_1 + b_1\Gamma_1+\cdots+ b_m \Gamma_m.
$$
By~\eqref{eq:imSigma} and~\eqref{eq:imGamma}, 
we obtain 
$$
[C]=\frac{1}{2q} \tileta_*\tileta^* ([C]) \in V.
$$
Therefore  $V\subset H^2(\tilS_d)$ is equal to the linear subspace spanned by 
the classes of all curves.
Combining this fact with Corollary~\ref{cor:supersingular},
we obtain  $V=H^2(\tilS_d)$.
\section{Supersingular $K3$ surfaces}
In this section, we prove Proposition~\ref{prop:ssK3}.
First, we recall some facts on supersingular $K3$ surfaces.
Let $Y$ be a supersingular $K3$ surface in characteristic $p$,
and let $\NS(Y)$ denote its N\'eron-Severi lattice,
which is an even hyperbolic lattice of rank $22$.
Artin~\cite{MR0371899} showed that the discriminant of $\NS(Y)$ is written as $-p^{2\sigma}$,
where $\sigma$ is a positive integer $\le 10$.
This integer $\sigma$ is called the \emph{Artin invariant} of $Y$.
Ogus~\cite{MR563467, MR717616} and Rudakov-Shafarevich~\cite{MR633161}
proved that, for each $p$, 
 a supersingular $K3$ surface with Artin invariant $1$ is unique up to isomorphisms.
 Let $X_p$ denote the supersingular $K3$ surface with Artin invariant $1$
 in characteristic $p$.
 It is known that $X_3$ is isomorphic to the Fermat quartic surface, 
 and that $X_5$ is isomorphic to the Fermat sextic double plane.
 (See, for example, \cite{KondoShimada} and~\cite{shimadapreprintchar5}, respectively.)
Therefore,
in order to prove Proposition~\ref{prop:ssK3},
it is enough to prove the following:
\begin{proposition}\label{prop:supersingK3}
Suppose that $(d, q+1)=(4,4)$ or $(2, 6)$.
Then, among the  curves on $\tilS_d$ listed in Proposition~\ref{prop:supersingular},
there exist $22$ curves whose classes together with the intersection pairing 
form a lattice of rank $22$ with discriminant $-p^2$.
\end{proposition}
\begin{proof}
Suppose that $p=q=3$ and $d=4$.
We put $\alpha:=\sqrt{-1}\in \F_9$, so that $\F_9:=\F_3 (\alpha)$.
Consider the projective space $\P^3$
with homogeneous coordinates $[w: x_0: x_1: x_2]$.
By Proposition~\ref{prop:defeq}, the surface $S_4$ is defined in $\P^3$ by an equation
$$
w^4=2(x_0^3 x_1 +x_0 x_1^3)-x_2^4-(x_2^2-x_1 x_0)^2.
$$
Hence the singular locus $\Sing(S_4)$ of $S_4$ consists of the three points
\begin{eqnarray*}
Q_0&:=&[0: 1: 1: 0]\quad (\textrm{located over\;\;}  \phi([1:\alpha])=\phi([1:-\alpha])\in B),\\
Q_1&:=&[0: 1: 2: 1]\quad (\textrm{located over\;\;}  \phi([1:1+\alpha])=\phi([1:1-\alpha])\in B),\\
Q_2&:=&[0: 1: 2: 2]\quad (\textrm{located over\;\;}  \phi([1:2+\alpha])=\phi([1:2-\alpha])\in B),
\end{eqnarray*}
and they are  rational double points of type $A_3$.
The minimal resolution $\rho: \tilS_4\to S_4$
is obtained by blowing up twice over each singular point $Q_a$ $(a\in \F_3)$.
The rational curves $l_P\spar{i}$ on $\tilS_4$ given in Proposition~\ref{prop:supersingular}
are the strict transforms of the following $40$ lines $\bar{L}_{\tau}\spar{\nu}$ in $\P^3$
contained in $S_4$,
where $\nu=0, \dots, 3$:
\begin{eqnarray*}
\bar{L}_{0}\spar{\nu}&:=&\{x_1=w-\alpha^{\nu} x_2=0\},\\
\bar{L}_{1}\spar{\nu}&:=&\{x_0+x_1-x_2=w-\alpha^{\nu} (x_2+x_0)=0\},\\
\bar{L}_{2}\spar{\nu}&:=&\{x_0+x_1+x_2=w-\alpha^{\nu} (x_2-x_0)=0\},\\
\bar{L}_{\infty}\spar{\nu}&:=&\{x_0=w-\alpha^{\nu} x_2=0\},\\
\bar{L}_{\pm \alpha}\spar{\nu}&:=&\{-x_0+x_1\pm\alpha x_2=w-\alpha^{\nu} x_2=0\},\\
\bar{L}_{1\pm \alpha}\spar{\nu}&:=&\{\pm \alpha x_0+x_1+(-1\pm \alpha)x_2=w-\alpha^{\nu} (x_2+x_0)=0\},\\
\bar{L}_{2\pm \alpha}\spar{\nu}&:=&\{\mp \alpha x_0+x_1+(1\pm \alpha)x_2=w-\alpha^{\nu} (x_2-x_0)=0\}.
\end{eqnarray*}
We denote by   $L_{\tau}\spar{\nu}$  the strict transform of $\bar{L}_{\tau}\spar{\nu}$ by $\rho$.
Note that the image of $\bar{L}_{\tau}\spar{\nu}$ by 
the covering morphism $S_4 \to \P^2$
is the line $l_{\phi([1:\tau])}$.
Note also that, if $\tau\in \F_3\cup\{\infty\}$,
then $\bar{L}_{\tau}\spar{\nu}$ is disjoint from $\Sing (S_4)$,
while if $\tau=a+b\alpha\in \F_9\setminus \F_3$ with $a\in \F_3$ and $b\in\F_3\setminus \{0\}=\{\pm 1\}$,
then $\bar{L}_{\tau}\spar{\nu}\cap \Sing (S_4)$ consists 
of a single point $Q_a$.
Looking at the minimal resolution $\rho$ over $Q_a$ explicitly, we see that 
the three exceptional $(-2)$-curves in $\tilS_4$ over $Q_a$ can be labeled
as $E_{a-\alpha}, E_{a}, E_{a+\alpha}$ in such a way that the following hold:
\begin{itemize}
\item 
$\intnumb{E_{a-\alpha}, E_a}=\intnumb{E_a, E_{a+\alpha}}=1$, $\intnumb{E_{a-\alpha},  E_{a+\alpha}}=0$.
\item
Suppose that $b\in \{\pm 1\}$.
Then $L_{a+b\alpha}\spar{\nu}$ intersects $E_{a+b\alpha}$,  and is disjoint from 
the other two irreducible components $E_{a}$ and $E_{a-b\alpha}$.
\item The four intersection points of  $L_{a+b\alpha}\spar{\nu}$ 
($\nu=0, \dots , 3$) and $E_{a+b\alpha}$ are distinct.
\end{itemize}
Using these,
we can calculate the intersection numbers among the $9+40$ curves
$E_{\tau}$ and $L_{\tau\sprime}\spar{\nu}$ ($\tau\in \F_9$, $\tau\sprime \in \F_9\cup\{\infty\}$, $\nu=0, \dots, 3$).
From among them, 
we choose the following $22$ curves:
\begin{eqnarray*}
&& E_{-\alpha},
E_{0},
E_{\alpha},
E_{1-\alpha},
E_{1},
E_{1+\alpha},
E_{2-\alpha},
E_{2},
E_{2+\alpha},\\
&& L_{0}\spar{0},
L_{0}\spar{1},
L_{0}\spar{2},
L_{0}\spar{3},
L_{1}\spar{0},
L_{1}\spar{1},
L_{2}\spar{0},
L_{2}\spar{1},
L_{\infty}\spar{1},\\
&& L_{-\alpha}\spar{0},
L_{-\alpha}\spar{1},
L_{1-\alpha}\spar{2},
L_{2-\alpha}\spar{0}.
\end{eqnarray*}
Their intersection numbers are calculated as 
 in Table~\ref{table:M22char3}.
 We can easily check that 
this matrix is of determinant $-9$.
Therefore the Artin invariant of $\tilS_4$ is $1$.
\begin{table}
{\small
$$
\left[ 
\begin {array}{c@{\hspace{2pt}}c@{\hspace{2pt}}c@{\hspace{2pt}}c@{\hspace{2pt}}c@{\hspace{2pt}}c@{\hspace{2pt}}c@{\hspace{2pt}}c@{\hspace{2pt}}c@{\hspace{2pt}}c@{\hspace{2pt}}c@{\hspace{2pt}}c@{\hspace{2pt}}c@{\hspace{2pt}}c@{\hspace{2pt}}c@{\hspace{2pt}}c@{\hspace{2pt}}c@{\hspace{2pt}}c@{\hspace{2pt}}c@{\hspace{2pt}}c@{\hspace{2pt}}c@{\hspace{2pt}}c@{\hspace{2pt}}} 
-2&1&0&0&0&0&0&0&0&0&0&0&0&0&0&0&0&0&1&1&0&0\\
1&-2&1&0&0&0&0&0&0&0&0&0&0&0&0&0&0&0&0&0&0&0\\
0&1&-2&0&0&0&0&0&0&0&0&0&0&0&0&0&0&0&0&0&0&0\\
0&0&0&-2&1&0&0&0&0&0&0&0&0&0&0&0&0&0&0&0&1&0\\
0&0&0&1&-2&1&0&0&0&0&0&0&0&0&0&0&0&0&0&0&0&0\\
0&0&0&0&1&-2&0&0&0&0&0&0&0&0&0&0&0&0&0&0&0&0\\
0&0&0&0&0&0&-2&1&0&0&0&0&0&0&0&0&0&0&0&0&0&1\\
0&0&0&0&0&0&1&-2&1&0&0&0&0&0&0&0&0&0&0&0&0&0\\
0&0&0&0&0&0&0&1&-2&0&0&0&0&0&0&0&0&0&0&0&0&0\\
0&0&0&0&0&0&0&0&0&-2&1&1&1&0&0&0&0&0&1&0&0&0
\\0&0&0&0&0&0&0&0&0&1&-2&1&1&0&0&0&0&1&0&1&0&0
\\0&0&0&0&0&0&0&0&0&1&1&-2&1&1&0&1&0&0&0&0&0&0
\\0&0&0&0&0&0&0&0&0&1&1&1&-2&0&1&0&1&0&0&0&1&1
\\0&0&0&0&0&0&0&0&0&0&0&1&0&-2&1&0&0&0&0&1&0&0
\\0&0&0&0&0&0&0&0&0&0&0&0&1&1&-2&0&0&1&0&0&0&1
\\0&0&0&0&0&0&0&0&0&0&0&1&0&0&0&-2&1&0&0&0&0&1
\\0&0&0&0&0&0&0&0&0&0&0&0&1&0&0&1&-2&1&1&0&1&0
\\0&0&0&0&0&0&0&0&0&0&1&0&0&0&1&0&1&-2&0&1&0&0
\\1&0&0&0&0&0&0&0&0&1&0&0&0&0&0&0&1&0&-2&0&0&0
\\1&0&0&0&0&0&0&0&0&0&1&0&0&1&0&0&0&1&0&-2&1&1
\\0&0&0&1&0&0&0&0&0&0&0&0&1&0&0&0&1&0&0&1&-2&0
\\0&0&0&0&0&0&1&0&0&0&0&0&1&0&1&1&0&0&0&1&0&-2
\end {array} \right] 
$$
}
\caption{Gram matrix of $\NS(\tilS_4)$ for $q=3$}\label{table:M22char3}
\end{table}
\begin{table}
{\small
$$
\left[ \begin {array}{c@{\hspace{2pt}}c@{\hspace{2pt}}c@{\hspace{2pt}}c@{\hspace{2pt}}c@{\hspace{2pt}}c@{\hspace{2pt}}c@{\hspace{2pt}}c@{\hspace{2pt}}c@{\hspace{2pt}}c@{\hspace{2pt}}c@{\hspace{2pt}}c@{\hspace{2pt}}c@{\hspace{2pt}}c@{\hspace{2pt}}c@{\hspace{2pt}}c@{\hspace{2pt}}c@{\hspace{2pt}}c@{\hspace{2pt}}c@{\hspace{2pt}}c@{\hspace{2pt}}c@{\hspace{2pt}}c@{\hspace{2pt}}}  
-2&0&0&0&0&0&0&0&0&0&0&0
&1&0&0&0&0&0&0&0&0&0\\0&-2&0&0&0&0&0&0&0&0&0&1&0&1&0
&0&0&0&0&0&0&0\\0&0&-2&0&0&0&0&0&0&0&0&0&0&0&1&0&0&0
&0&0&0&0\\0&0&0&-2&0&0&0&0&0&0&0&0&0&0&0&1&0&0&0&0&0
&0\\0&0&0&0&-2&0&0&0&0&0&0&0&0&0&0&0&0&1&0&0&0&0
\\0&0&0&0&0&-2&0&0&0&0&0&0&0&0&0&0&0&0&0&1&0&0
\\0&0&0&0&0&0&-2&0&0&0&0&0&0&0&0&0&0&0&0&0&1&0
\\0&0&0&0&0&0&0&-2&0&0&0&0&0&0&0&0&0&0&0&0&0&1
\\0&0&0&0&0&0&0&0&-2&3&1&1&0&1&1&0&0&1&1&1&0&1
\\0&0&0&0&0&0&0&0&3&-2&0&0&1&0&0&1&1&0&0&0&1&0
\\0&0&0&0&0&0&0&0&1&0&-2&0&0&0&1&1&0&1&0&0&0&1
\\0&1&0&0&0&0&0&0&1&0&0&-2&0&0&1&1&1&1&0&1&1&0
\\1&0&0&0&0&0&0&0&0&1&0&0&-2&0&0&0&1&1&1&1&0&1
\\0&1&0&0&0&0&0&0&1&0&0&0&0&-2&0&1&0&0&1&0&0&0
\\0&0&1&0&0&0&0&0&1&0&1&1&0&0&-2&1&1&0&1&0&1&1
\\0&0&0&1&0&0&0&0&0&1&1&1&0&1&1&-2&1&0&0&1&0&0
\\0&0&0&0&0&0&0&0&0&1&0&1&1&0&1&1&-2&1&1&1&0&0
\\0&0&0&0&1&0&0&0&1&0&1&1&1&0&0&0&1&-2&0&0&0&0
\\0&0&0&0&0&0&0&0&1&0&0&0&1&1&1&0&1&0&-2&0&1&0
\\0&0&0&0&0&1&0&0&1&0&0&1&1&0&0&1&1&0&0&-2&1&1
\\0&0&0&0&0&0&1&0&0&1&0&1&0&0&1&0&0&0&1&1&-2&1
\\0&0&0&0&0&0&0&1&1&0&1&0&1&0&1&0&0&0&0&1&1&-2
\end {array} \right] 
$$
}
\caption{Gram matrix of $\NS(\tilS_2)$ for $q=5$}\label{table:M22char5}
\end{table}
\par
\medskip
The proof for the case  $p=q=5$ and $d=2$ is similar.
We put $\alpha:=\sqrt{2}$ so that $\F_{25}=\F_5(\alpha)$.
In the weighted projective space $\P(3,1,1,1)$
with homogeneous coordinates $[w:x_0:x_1:x_2]$, 
the surface $S_2$ for $p=q=5$ is defined by
$$
w^2=2 (x_0^5 x_1+ x_0 x_1^5)-x_2^6-(x_2^2+x_0 x_1)^3.
$$
The singular locus $\Sing(S_2)$ consists of ten ordinary nodes
$$
Q_{\{a+b\alpha, a-b\alpha\}}\quad (a\in \F_5, b\in \{1,2\})
$$
located over the nodes $\phi([1:a+b\alpha])=\phi([1:a-b\alpha])$
of the branch curve  $B$.
Let $E_{\{a+b\alpha, a-b\alpha\}}$ denote the exceptional $(-2)$-curve in $\tilS_2$
over $Q_{\{a+b\alpha, a-b\alpha\}}$ by the minimal resolution.
As the $22$ curves, 
we choose the following eight exceptional $(-2)$-curves
\begin{eqnarray*}
&&E_{{ \left\{ -\alpha,\alpha \right\} }},\quad
E_{{ \left\{ -2\,\alpha,2\,\alpha \right\} }},\quad
E_{{ \left\{ 1-\alpha,1+\alpha \right\} }}, \quad
E_{{ \left\{ 1-2\,\alpha,1+2\,\alpha \right\} }},\\
&&E_{{ \left\{ 2-\alpha,2+\alpha \right\} }}, \quad
E_{{ \left\{ 3-2\,\alpha,3+2\,\alpha \right\} }}, \quad
E_{{ \left\{ 4-\alpha,4+\alpha \right\} }}, \quad
E_{{ \left\{ 4-2\,\alpha,4+2\,\alpha \right\} }},
\end{eqnarray*}
and the strict transforms of the following $14$ curves on $S_2$: 
\begin{eqnarray*}
&& \{\;\;x_{{1}}\;\;=\;\;w-2\,\alpha {x_{{2}}}^{3}\;\;=\;\;0\;\;\}, \\
&& \{\;\; x_{{1}}\;\;=\;\; w+2\,\alpha {x_{{2}}}^{3}\;\;=\;\;0\;\;\}, \\
&& \{\;\;x_{{0}}+ x_{{1}}+4\,x_{{2}}\;\;=\;\;w+2\,\alpha  \left( 3\,x_{{0}}+x_{{2}} \right) ^{3}\;\;=\;\;0\;\;\},\\
&&\{\;\; 3\,x_{{0}}+x_{{1}}+3\,\alpha x_{{2}}\;\;=\;\;
 w-2\,\alpha {x_{{2}}}^{3}\;\;=\;\;0\;\;\},\\
&&\{\;\;2\,x_{{0}}+ x_{{1}}+4\,\alpha x_{{2}}\;\;=\;\;
 w+2\,\alpha {x_{{2}}}^{3}\;\;=\;\;0\;\;\},\\
&&\{\;\; 3\,x_{{0}}+x_{{1}}+2\,\alpha x_{{2}}+3\,x_{{0}}\;\;=\;\;
 w-2\,\alpha {x_{{2}}}^{3}\;\;=\;\;0\;\;\},\\
&&\{\;\; \left( 3+3\,\alpha  \right) x_{{0}}+x_{{1}}+ \left( 4+\alpha  \right) x_{{2}}\;\;=\;\;
 w+2\,\alpha  \left( 3\,x_{{0}}+x_{{2}} \right) ^{3}\;\;=\;\;0\;\;\},\\
&&\{\;\; \left( 4+\alpha \right) x_{{0}}+ x_{{1}}+ \left( 4+2\,\alpha  \right) x_{{2}}\;\;=\;\;
 w+2\,\alpha  \left( 3\,x_{{0}}+x_{{2}} \right) ^{3}\;\;=\;\;0\;\;\},\\
&&\{\;\; \left( 2+3\,\alpha  \right) x_{{0}}+ x_{{1}}+ \left( 3+3\,\alpha  \right) x_{{2}}\;\;=\;\;
 w-2\,\alpha  \left( x_{{0}}+x_{{2}} \right) ^{3}\;\;=\;\;0\;\;\},\\
&&\{\;\; \left( 1+\alpha \right) x_{{0}}+x_{{1}}+ \left( 3+\alpha  \right) x_{{2}}\;\;=\;\;
 w-2\,\alpha  \left( x_{{0}}+x_{{2}} \right) ^{3}\;\;=\;\;0\;\;\},\\
&&\{\;\; \left( 1+\alpha \right) x_{{0}}+ x_{{1}}+ \left( 2+4\,\alpha  \right) x_{{2}}\;\;=\;\;
 w-2\,\alpha  \left( x_{{2}}+4\,x_{{0}} \right) ^{3}\;\;=\;\;0\;\;\},\\
&&\{\;\; \left( 2+3\,\alpha  \right) x_{{0}}+x_{{1}}+ \left( 2+2\,\alpha  \right) x_{{2}}\;\;=\;\;
 w+2\,\alpha  \left( x_{{2}}+4\,x_{{0}} \right) ^{3}\;\;=\;\;0\;\;\},\\
&&\{\;\; \left( 3+3\,\alpha \right) x_{{0}}+ x_{{1}}+ \left( 1+4\,\alpha  \right) x_{{2}}\;\;=\;\;
 w-2\,\alpha  \left( x_{{2}}+2\,x_{{0}} \right) ^{3}\;\;=\;\;0\;\;\},\\
&&\{\;\; \left( 4+4\,\alpha \right) x_{{0}}+x_{{1}}+ \left( 1+2\,\alpha  \right) x_{{2}}\;\;=\;\;
 w-2\,\alpha  \left( x_{{2}}+2\,x_{{0}} \right) ^{3}\;\;=\;\;0\;\;\}.
\end{eqnarray*}
Their intersection matrix is given in Table~\ref{table:M22char5}.
It is of determinant $-25$.
Therefore the Artin invariant of $\tilS_2$ is $1$.
\end{proof}

\begin{remark}
In the case $q=5$,
the Ballico-Hefez curve $B$ is one of the  sextic plane curves studied classically by Coble~\cite{MR1506391}.
\end{remark}

\bibliographystyle{plain}

\def\cftil#1{\ifmmode\setbox7\hbox{$\accent"5E#1$}\else
  \setbox7\hbox{\accent"5E#1}\penalty 10000\relax\fi\raise 1\ht7
  \hbox{\lower1.15ex\hbox to 1\wd7{\hss\accent"7E\hss}}\penalty 10000
  \hskip-1\wd7\penalty 10000\box7} \def\cprime{$'$} \def\cprime{$'$}
  \def\cprime{$'$} \def\cprime{$'$}

\end{document}